\documentclass[11pt]{amsart} 

\usepackage{amsmath, amsthm, amssymb, fullpage}

\usepackage{enumerate, paralist}

\usepackage{hyperref}

\newcommand{\bP}{\mathbb{P}}
\newcommand{\sP}{\mathsf{P}}
\newcommand{\sH}{\mathsf{H}}
\newcommand{\widevec}[1]{\overrightarrow{#1}}
\DeclareMathOperator{\cO}{\mathcal{O}}
\DeclareMathOperator{\cM}{\mathcal{M}}
\newcommand{\cS}{\mathcal{S}}   
\DeclareMathOperator{\PGL}{PGL}
\DeclareMathOperator{\SL}{SL}
\DeclareMathOperator{\dep}{depth}
\newcommand{\bN}{\mathbb{N}}
\DeclareMathOperator{\hypRes}{hypRes}
\newcommand{\bD}{\mathbb{D}}
\newcommand{\bC}{\mathbb{C}}
\newcommand{\bL}{\mathbb{L}}
\DeclareMathOperator{\GCD}{GCD}

\newcommand{\bR}{\mathbb{R}}
\DeclareMathOperator{\ordRes}{ordRes}
\DeclareMathOperator{\Res}{Res}
\newcommand{\rd}{\mathrm{d}}

\DeclareMathOperator{\supp}{supp}

\theoremstyle{plain}
\newtheorem{theorem}{Theorem}[section]

\newtheorem{mainth}{Theorem}

\newtheorem{proposition}[theorem]{Proposition}
\theoremstyle{definition}
\newtheorem{definition}[theorem]{Definition} 
 
\newtheorem*{notation}{Notation}
\newtheorem*{acknowledgement}{Acknowledgement}

\theoremstyle{remark}
\newtheorem{fact}[theorem]{Fact}

\numberwithin{equation}{section}

\makeatletter
\@namedef{subjclassname@2020}{\textup{2020} Mathematics Subject Classification}
\makeatother

\begin{document}

\title[intrinsic reductions and depths]{The intrinsic reductions and the intrinsic depths in non-archimedean dynamics}

\author{Y\^usuke Okuyama}
\address{Division of Mathematics, Kyoto Institute of Technology, Sakyo-ku, Kyoto 606-8585 JAPAN}
\email{okuyama@kit.ac.jp}
 


\date{\today}

\begin{abstract}
In this short paper, we aim at giving a more conceptual and simpler proof of
Rumely's moduli theoretic characterization of 
type II minimal locus of the resultant function 
$\ordRes_\phi$ 
on the Berkovich hyperbolic space
for a rational function $\phi$ on $\bP^1$
defined over an algebraically closed and complete field that is equipped with a non-trivial and
non-archimedean absolute value, 
and also aim at giving a much simpler and
more natural proof of 
a degenerating limit theorem,
in an improved form after DeMarco--Faber, 
for the family of
the unique maximal entropy measures on $\bP^1(\bC)$
associated to a meromorphic family of complex rational functions.
We introduce the intrinsic reduction of 
a non-archimedean rational function $\phi$
at each point in the Berkovich projective line and 
its directionwise intrinsic depths, which 
are suitable notions for the above aims and
defined in terms of the tree and analytic structures of the Berkovich projective line. 
Then we establish two theorems in non-archimedean dynamics, both of which
play key roles in the above aims.
\end{abstract}

\maketitle

\section{Introduction}
\label{sec:intro}

Let $K$ be an algebraically closed 
and complete field that is equipped with a non-trivial
and non-archimedean absolute value $|\cdot|$. 
The classical action of a non-constant
rational function $\phi(z)\in K(z)$ on 
the projective line $\bP^1=\bP^1(K)$ extends, as an analytic endomorphism,
to the Berkovich projective line $\sP^1=\sP^1_K$, which is a compact augmentation of $\bP^1$ 
and has both an analytic structure in Berkovich's sense and a profinite tree structure in Jonsson's sense, and this extended action of $\phi$
preserves both $\bP^1$ and
the Berkovich upper half space 
$\sH^1=\sH^1(K):=\sP^1\setminus\bP^1$; 
the extended action to $\sP^1_K$
of the projective transformation group 
$\PGL(2,K)$ of $\bP^1$ 
is transitive on $\sH^1$. 
For the details on the properties of $\sP^1$ (and Berkovich curves)
and non-archimedean dynamics on $\sP^1$,
we refer to \cite{BR10, BenedettoBook, FR09, Jonsson15}.
We denote the direction (or tangent) space of $\sP^1$
at a point $\xi\in\sP^1$ by $T_{\xi}\sP^1$, and also denote the direction
$\vec{v}\in T_{\xi}\sP^1$ pointing a point $\xi'\in\sP^1\setminus\{\xi\}$ 
by $\widevec{\xi\xi'}$, which is identified with the connected component
$U(\vec{v})$ of $\sP^1\setminus\{\xi\}$ containing $\xi'$. 
The analytic endomorphism $\phi$ of $\sP^1$ is $\deg\phi$ to $1$
taking into account the local degree 
$\deg_\xi(\phi)$ of $\phi$ 
at each point $\xi\in\sP^1$,
and induces
the pullback operator $\phi^*$ from the space of Radon measures on $\sP^1$
into itself so that the pullback 
under $\phi$ of the Dirac measure $\delta_\xi$ on $\sP^1$ at a point
$\xi\in\sP^1$ is the sum of the Dirac measures $\delta_{\xi'}$
on $\sP^1$, $\xi'\in\phi^{-1}(\xi)$,
taking into account the local degree of $\phi$ at each $\xi'$.

Let us introduce the following notions, which are
defined in terms of the tree 
and analytic structures of $\sP^1$.

\begin{definition}\label{th:pointwise}
The {\em intrinsic reduction} $\tilde{\phi}_\xi$
of $\phi$ at a point $\xi\in\sP^1$ is 
a selfmap of $T_\xi\sP^1$ defined as 
\begin{gather*}
 \tilde{\phi}_\xi\begin{cases}
		    :=\phi_{*,\xi} & \text{if }\phi(\xi)=\xi,\\
		    :\equiv\widevec{\xi(\phi(\xi))} & \text{otherwise},
		   \end{cases}
\end{gather*}
where $\phi_*=\phi_{*,\xi}:T_\xi\sP^1\to T_{\phi(\xi)}\sP^1$ is the 
(non-constant) tangent map
induced by $\phi$ at $\xi$, and its {\em intrinsic depth}
at each direction $\vec{v}\in T_\xi\sP^1$ is the non-negative 
integer defined as
\begin{gather}
 \dep_{\vec{v}}\tilde{\phi}_\xi:=(\phi^*\delta_{\xi})\bigl(U(\vec{v})\bigr)\in\{0,\ldots,\deg\phi\}.\label{eq:geomdepth}
\end{gather}
\end{definition}

The above notions are closely related to more algebraic
notions in non-archimedean dynamics;
when the point $\xi$ is the Gauss or canonical point $\xi_g$ in $\sH^1$,
the direction space $T_{\xi_g}\sP^1$ is identified with
$\bP^1(k)$ by the bijection $\widevec{\xi_g a}\leftrightarrow\hat{a}$,
and the intrinsic reduction $\tilde{\phi}_{\xi_g}$
is with the (algebraic) reduction $\tilde{\phi}(\zeta)\in k(\zeta)$ of $\phi$ modulo $K^{\circ\circ}$ (\cite{Juan03}, see also \cite[Corollary 9.27]{BR10}), where $K^{\circ\circ}:=\{z\in K:|z|<1\}$
is the unique maximal ideal of the ring $K^{\circ}:=\{z\in K:|z|\le 1\}$
of $K$-integers\footnote{Here we adopted the notations $K^{\circ},K^{\circ\circ}$, which seems as popular as the notations $\cO_K,\cM_K$, respectively.}$k=k_K:=K^\circ/K^{\circ\circ}$ is the residual field of $K$,
and the point $\hat{a}\in\bP^1(k)$ is the reduction of a point $a\in\bP^1(K)$ modulo $K^{\circ\circ}$. Moreover, $\deg_{\xi_g}(\phi)=\deg\tilde{\phi}$, and 
the intrinsic depth $\dep_{\vec{v}}\tilde{\phi}_{\xi_g}$
at $\vec{v}=\widevec{\xi_g a}$ equals the (algebraic) depth
$\dep_{\hat{a}}\hat{\phi}$ at $\hat{a}$   
of the coefficient reduction 
$\hat{\phi}\in\bP^{2(\deg\phi)+1}(k)$ of $\phi$ 
modulo $K^{\circ\circ}$
(see Section \ref{sec:slope} for the definitions).
When $\xi\in\sH^1$ 
(e.g.\ when $\xi=\xi_g$ as the above),
the notions of the intrinsic reduction of $\phi$ at $\xi$ 
and its intrinsic depths are also
related to the reduced version $\phi_\xi^*$ at $\xi$
of the above pullback operator $\phi^*$ for measures on $\sP^1$
(introduced in \cite[\S 4.3]{DF14}, and the definition is mentioned in Section \ref{sec:degenerating}),  
which is an operator from the space $M(\{\xi,\phi(\xi)\})$
to the space $M(\{\xi\})$ of measures on the factored spaces
$\sP^1/S(\{\xi,\phi(\xi)\}),\sP^1/S(\{\xi\})$
induced respectively by the partitions 
$S(\{\xi,\phi(\xi)\}),S(\{\xi\})$ of $\sP^1$,
setting the partition
\begin{gather*}
 S(\Gamma):=\{\{\xi\}:\xi\in\Gamma\}
\sqcup\{\text{connected components of }\sP^1\setminus\Gamma\}
\end{gather*} 
of $\sP^1$ for any finite subset $\Gamma$ in $\sP^1$ 
and equipping the factored space $\sP^1/S(\Gamma)$ with the discrete topology; 
for example, for any direction
$\vec{v}\in T_\xi\sP^1=S(\{\xi\})\setminus\{\{\xi\}\}$
at the point $\xi$,
\begin{gather}
\bigl(\phi_\xi^*\delta_{\{\xi\}}\bigr)(\{U(\vec{v})\})=\dep_{\vec{v}}{\tilde{\phi}_\xi},
\end{gather}
using the non-archimedean argument principle
\eqref{eq:FKN} 
(\cite[Proposition 3.10]{Faber13topologyI})
stated below.
\begin{notation}
 Above and below,
 let $\delta_U$ denote the Dirac measure on the space $\sP^1/S(\Gamma)$ 
 at each point $U\in \sP^1/S(\Gamma)$.
For any $\xi,\xi'\in\sP^1$,
let $[\xi,\xi']$ denote the closed interval in 
(the tree) $\sP^1$ between $\xi,\xi'\in\sP^1$.
\end{notation}

Our aim is to contribute to studying moduli problems on
reductions of non-archimedean dynamics
and degeneration problems on meromorphic families of complex dynamics, 
by establishing two theorems
on Berkovich dynamics of $\phi$.

From now on, we assume $d:=\deg\phi>1$.

\subsection{Reduction theoretic characterization of the minimum locus for the hyperbolic resultant function}\label{sec:moduli}
In \cite[Definition 1.1]{Okugeometric}, the hyperbolic resultant function\footnote{Here we followed the naming in \cite{NO24}. The naming ``crucial function'' as in \cite{Okugeometric} would also be adequate.}  
\begin{gather} 
 \hypRes_\phi(\xi):=
\frac{\rho(\xi,\xi_g)}{2}
+\frac{\rho(\xi,\phi(\xi)\wedge_{\xi_g}\xi)-\int_{\sP^1}\rho(\xi_g,\xi\wedge_{\xi_g}\cdot)(\phi^*\delta_{\xi_g})(\cdot)}{d-1}\in\bR
\label{eq:hypres}
\end{gather}
on $\sH^1$ for $\phi$ has been introduced, 
where and below
$\rho$ is the hyperbolic metric on $\sH^1$
and, for any points 
$\xi,\xi',\xi''\in\sP^1$,
the point $\xi\wedge_{\xi''}\xi'$ is the unique point
simultaneously 
belonging to the intervals $[\xi,\xi''],[\xi',\xi'']$, and $[\xi,\xi']$
in $\sP^1$.
The function $\hypRes_\phi$ 
is indeed a proper, convex, and piecewise affine function on 
the Berkovich hyperbolic space $(\sH^1,\rho)$, which is an $\bR$-tree (\cite[Theorem 2]{Okugeometric}).
We denote by $\rd_{\vec{v}}=\frac{\rd}{\rd(\rho|[\xi,\xi'])}\bigl|_\xi$ 
the directional derivative operator on $(\sH^1,\rho)$
for a direction $\vec{v}=\widevec{\xi\xi'}\in T_\xi\sP^1$ at $\xi\in\sH^1$. 

Our principal result in this subsection is the following
reduction theoretic slope formula for $\hypRes_\phi$, which 
characterizes the minimum locus of $\hypRes_\phi$
in a reduction theoretic way.

\begin{mainth}\label{th:slope}
Let $\phi\in K(z)$ be a rational function on $\bP^1$ of degree $d>1$. Then
for every point $\xi\in\sH^1$ and every direction $\vec{v}\in T_\xi\sP^1$,
\begin{gather}
 \rd_{\vec{v}}\hypRes_\phi=\frac{1}{d-1}
\biggl(-\dep_{\vec{v}}\tilde{\phi}_{\xi}
+\frac{1}{2}\cdot\begin{cases}
		  d-1 & \text{if }\tilde{\phi}_\xi(\vec{v})=\vec{v}\\
		  d+1 & \text{otherwise}
		 \end{cases}\biggr);\label{eq:slopdep}
\end{gather}
in particular, the function $\hypRes_\phi$ takes its minimum
at $\xi$ (resp.\ uniquely at $\xi$) if and only if
for every $\vec{v}\in T_\xi\sP^1$,
\begin{gather}
\begin{cases}
 \dep_{\vec{v}}\tilde{\phi}_{\xi}\le\dfrac{d+1}{2}  
 \,\bigl(\text{resp.}\,\le\dfrac{d}{2}\bigr)
 \text{ and more strictly}\vspace*{5pt}\\
 \dep_{\vec{v}}\tilde{\phi}_{\xi}<\dfrac{d}{2} 
 \,\bigl(\text{resp.}\,<\dfrac{d-1}{2}\bigr) 
 \text{ if also }\tilde{\phi}_\xi(\vec{v})=\vec{v}. 
\end{cases}\label{eq:charecterization}
\end{gather}
\end{mainth}

The proof of Theorem \ref{th:slope} is by a simple computation 
in Berkovich hyperbolic geometry. In \cite[(1.1)]{Okugeometric},
using the function $\hypRes_\phi$,
the following explicit, global, and geometric expression
\begin{gather}
 \ordRes_\phi
=2d(d-1)\hypRes_\phi-\log(
\bigl|\Res(\text{a degree $d$ minimal 
homogeneous lift of }\phi)\bigr|)
\label{eq:ordhyp}
\end{gather}
on $\sH^1$
of Rumely's resultant function $\ordRes_\phi$
(that Rumely defined
more algebraically and implicitly and studied in
\cite{Rumely13,Rumely17}) is given.
In particular the minimum loci of $\hypRes_\phi$ and of $\ordRes_\phi$
are identical. By this fact and the above mentioned 
identification of the intrinsic depths of $\tilde{\phi}_{\xi_g}$
with the (algebraic) depths of $\hat{\phi}$ modulo $K^{\circ\circ}$,
the characterization \eqref{eq:charecterization} of the minimum locus
of $\hypRes_\phi$ gives a more conceptual and simpler proof of
Rumely's moduli theoretic characterization
of the type II minimum locus of $\ordRes_\phi$ 
(see Section \ref{sec:slope}).
This found some application in
a recent hybrid analysis (see, e.g., \cite{FavreGong24}). 

\subsection{Equidistribution property for the intrinsic depths towards
the reduced canonical equilibrium measure}\label{sec:deglim}
The canonical equilibrium measure $\nu_\phi$ for $\phi$ on $\sP^1$
is the unique probability Radon measure $\nu$ on $\sP^1$ satisfying
both the balancing property 
\begin{gather}
 \phi^*\nu=d\cdot\nu\quad\text{on }\sP^1 
\end{gather} 
under $\phi$
and the non-exceptionality condition $\nu(E(\phi))=0$ for $\phi$, where
the at most countable subset
$E(\phi):=\{a\in\bP^1:\#\bigcup_{n\in\bN}\phi^{-n}(a)<+\infty\}$
in $\bP^1$
is called the (classical) exceptional set of $\phi$.
The existence of such $\nu_\phi$ is based on 
the equidistribution property
\begin{gather}
 \nu_\phi=\lim_{n\to\infty}\frac{(\phi^n)^*\delta_{\xi}}{d^n}\quad\text{weakly on }\sP^1,
\label{eq:equidist}
\end{gather} 
for any $\xi\in\sP^1\setminus E(\phi)$,
of the averaged iterated pullbacks of $\delta_\xi$ under $\phi$
(\cite{BR10, FR09, ChambertLoir06});
in particular, for every $\xi\in\sH^1$,
denoting by $\pi_\xi$ the projection from $\sP^1$
to the factored space $\sP^1/S(\{\xi\})$ induced by the partition
$S(\{\xi\})(=\{\{\xi\}\}\sqcup T_\xi\sP^1)$, the pushforward (or reduction) $(\pi_\xi)_*\nu_\phi$
of the measure $\nu_\phi$ under $\pi_\xi$ 
is a probability measure on $\sP^1/S(\{\xi\})$ 
written as the sum of at most countably many
atoms, and moreover 
\begin{gather*}
 ((\pi_\xi)_*\nu_\phi)(\{\{\xi\}\})(=\nu_\phi(\{\xi\}))>0 
\end{gather*}
if and only if $\xi$ is totally invariant under $\phi$
in that $\phi^{-1}(\xi)=\{\xi\}$.

Our principal result in this subsection is the following 
equidistribution property for the intrinsic depths.

\begin{mainth}\label{th:unicity}
Let $\phi\in K(z)$ be a rational function on $\bP^1$ of degree $d>1$.
Then for every $\xi\in\sH^1$ not totally invariant under $\phi$, we have
\begin{gather}
 \lim_{n\to\infty}\frac{1}{d^n}\sum_{\vec{v}\in T_\xi\sP^1}\bigl(\dep_{\vec{v}}(\widetilde{\phi^n})_\xi\bigr)\delta_{U(\vec{v})}
 =(\pi_\xi)_*\nu_\phi
\label{eq:surplusconv}
\end{gather}
weakly on $\sP^1/S(\{\xi\})$. 
\end{mainth}
Theorem \ref{th:unicity} is a simple consequence of 
the equidistribution property \eqref{eq:equidist}
as well as the inner regularity of $\nu_\phi$ and Fatou's lemma.
When the non-archimedean
field $K$ is the Levi-Civita field $\bL$,
which is the completion of an algebraic closure of the Laurent series field $\bC((t))$
equipped with a $t$-adic norm $|\cdot|_r$ normalized as $|t|_r=r$
(fixing $r\in(0,1)$ once and for all), 
letting $\cO(\bD)$ be the ring of
holomorphic functions on the open unit disk $\bD$ in $\bC$
and regarding $\cO(\bD)[t^{-1}]$ 
as a subring in $\bL$, any rational function
$f(z)\in(\cO(\bD)[t^{-1}])(z)$ of degree $d>1$
determines not only an element in $\bL(z)$ of degree $d$ but also
a holomorphic family 
$(f_t)_{0<|t|\ll 1}$ of complex rational 
functions $f_t$ on $\bP^1(\bC)$
of degree $d$ (where $|\cdot|$ denotes the Euclidean norm on $\bC$), 
and for each such a specialization $f_t$ of $f$,
there is
a canonical equilibrium measure
\begin{gather*}
 \mu_t:=\mu_{f_t} 
\end{gather*}
on $\bP^1(\bC)$ which is indeed
the unique maximal entropy measure for $f_t$.
Based on Theorem \ref{th:unicity},
we give an improvement of 
the degenerating limit theorem for $(\mu_t)_{0<|t|\ll 1}$
(see Theorem \ref{th:degenerating}
in Section \ref{sec:degenerating}), which has been deepen in a recent study of degenerating iteration maps (\cite{KN24}).

\subsection*{Remarks on notations}
The use of notations and terminologies in the Berkovich dynamics theory
is still not uniform in the literature.
We adopt Favre--Rivera-Letelier's notation $\sP^1$ for the Berkovich projective line and the notation $\xi$, which can be thought of as a variant of Baker--Rumely's notation $\zeta$, for a point in $\sP^1$
(while Favre--Rivera-Letelier's is $\cS$).
Again, several different notations are
used for the Berkovich upper half space, for which 
we adopted Favre--Rivera-Letelier's notation $\sH^1$. 
On the naming, the hyperbolic metric $\rho$
is also called the big model metric. 
 
\section{Proof of Theorem \ref{th:slope} and its application to
a moduli problem on reductions}\label{sec:slope}
Let $\phi\in K(z)$ be a rational function on $\bP^1$ of degree $d>1$.

\begin{proof}[Proof of Theorem \ref{th:slope}]
 By the difference (or base point change) formula
\begin{gather*}
 \hypRes_\phi(\xi')-\hypRes_\phi(\xi)
=
\frac{\rho(\xi',\xi)}{2}
+\frac{\rho(\xi',\phi(\xi')\wedge_{\xi}\xi')-\int_{\sP^1}\rho(\xi',\xi'\wedge_{\xi}\cdot\,)(\phi^*\delta_{\xi})(\cdot)}{d-1},
\end{gather*}
$\xi',\xi\in\sH^1$ (\cite[(1.2)]{Okugeometric}), 
for every direction $\vec{v}=\widevec{\xi\xi'}\in T_\xi\sP^1$ at 
a point $\xi\in\sH^1$, we have
 \begin{gather*}
 \rd_{\vec{v}}\hypRes_\phi
 =
  \frac{1}{2}
+\frac{1}{d-1}
\biggl(\rd_{\vec{v}}\bigl(\xi'\mapsto\rho(\xi',\phi(\xi')\wedge_\xi\xi')\bigr)
 -\rd_{\vec{v}}\Bigl(\xi'\mapsto\int_{\sP^1}\rho(\xi',\xi'\wedge_{\xi}\cdot\,)(\phi^*\delta_{\xi})\Bigr)\biggr).
 \end{gather*}
On one hand, we have the equality 
\begin{gather*}
 \rd_{\vec{v}}\Bigl(\xi'\mapsto\int_{\sP^1}\rho(\xi',\xi'\wedge_{\xi}\cdot)(\phi^*\delta_{\xi})\Bigr)=(\phi^*\delta_\xi)(U(\vec{v}))
 =:\dep_{\vec{v}}\tilde{\phi}_{\xi} 
\end{gather*}
from the computation in \cite[(4.2)]{Okugeometric}.
On the other hand, we also have
\begin{gather*}
 \rd_{\vec{v}}\bigl(\xi'\mapsto\rho(\xi',\phi(\xi')\wedge_\xi\xi')\bigr)=
 \begin{cases}
 0 &\text{if }\tilde{\phi}_\xi(\vec{v})=\vec{v},\\
 1 &\text{otherwise};
 \end{cases}
\end{gather*}
for, if $\xi'\in U(\vec{v})$ is close enough to $\xi$, then we also
compute
 \begin{gather*}
 \rho\bigl(\,\cdot\,,\phi(\cdot)\wedge_{\xi}\,\cdot\,\bigr)
=
 \begin{cases}
 0 &
 \text{either if }\phi(\xi)=\xi
\text{ and }\phi_*\vec{v}=\vec{v}
 \text{ or if }\phi(\xi)\neq\xi\text{ and }\widevec{\xi\phi(\xi)}=\vec{v}\\
 \rho(\,\cdot\,,\xi)&
 \text{either if }\phi(\xi)=\xi
\text{ and }\phi_*\vec{v}\neq\vec{v}
 \text{ or if }\phi(\xi)\neq\xi\text{ and }\widevec{\xi\phi(\xi)}\neq\vec{v}
 \end{cases}
 \end{gather*} 
 on $[\xi,\xi']\setminus\{\xi\}$ (\cite[(4.1)]{Okugeometric}). 
 Now the proof of \eqref{eq:slopdep} 
 is complete. 

The formula \eqref{eq:slopdep} together with 
the convexity and piecewise affineness of $\hypRes_\phi$
on $(\sH^1,\rho)$ concludes \eqref{eq:charecterization}.
\end{proof}

Let us now recall the statement of
Rumely's moduli theoretic characterization of the type II minimum locus
of Rumely's resultant function $\ordRes_\phi$ on $\sH^1$
mentioned in Subsection \ref{sec:moduli},
and also fix some notations and include some details. 

\begin{theorem}[{\cite[Theorem C]{Rumely17}}]\label{th:ss}
 Let $\phi\in K(z)$ be a rational function on $\bP^1$ of degree $d>1$.
 Then for every type II point $\xi\in\sH^1$, the following 
 statements are equivalent;
 \begin{inparaenum}[$(1)$]
 \item $\ordRes_\phi$ takes its minimum at $\xi$ (resp.\ uniquely at $\xi$).
 \item $\widehat{(M\circ\phi\circ M^{-1})}\in(\bP^{2d+1})^{\operatorname{ss}}(k)$ (resp.\ $\in(\bP^{2d+1})^{\operatorname{s}}(k)$) writing $\xi=M(\xi_g)$
for a unique element $[M]=\PGL(2,K^\circ)M$ in the left orbit space
$\PGL(2,K^\circ)\backslash\PGL(2,K)$.
 \end{inparaenum}
\end{theorem}
The quotient epimorphism 
$K^{\circ}\ni a\mapsto\hat{a}\in k:=K^\circ/K^{\circ\circ}$
extends to a map from $\bP^1(K)=K\sqcup\{\infty\}$ to $\bP^1(k)=k\sqcup\{\infty\}$ by setting $\hat{z}:\equiv\infty\in\bP^1(k)$ for any $z\in\bP^1(K)\setminus K^\circ$.
Writing 
\begin{gather*}
\phi(z)=\frac{\sum_{i=0}^db_iz^i}{\sum_{i=0}^da_iz^i},
\end{gather*} 
the {\em coefficient reduction} of $\phi$
 modulo $K^{\circ\circ}$
 is the point
 $\hat{\phi}:=[\widehat{a_0}:\cdots:\widehat{a_d}:\widehat{b_0}:\cdots:\widehat{b_d}]\in\bP^{2d+1}(k)$ 
 induced by a 
minimal representative $(a_0,\ldots,a_d,b_0,\ldots,b_d)\in 
(K^{\circ})^{2d+2}\setminus(K^{\circ\circ})^{2d+2}$
of the point $[a_0:\cdots:a_d:b_0:\cdots:b_d]\in\bP^{2d+1}(K)$. 
Setting also 
\begin{gather*}
 H_{\hat{\phi}}(X_0,X_1)
:=\GCD\Bigl(\sum_{i=0}^d\widehat{b_i}X_0^{d-i}X_1^i,\sum_{i=0}^d\widehat{a_i}X_0^{d-i}X_1^i\Bigr)
\end{gather*}
in $k[X_0,X_1]\setminus\{0\}$,
the {\em reduction} of $\phi$ modulo $K^{\circ\circ}$ is a 
(possibly constant) rational function
\begin{gather*}
 \tilde{\phi}(\zeta):=\frac{\sum_{i=0}^d\widehat{b_i}\zeta^i/H_{\hat{\phi}}(1,\zeta)}{\sum_{i=0}^d\widehat{a_i}\zeta^i/H_{\hat{\phi}}(1,\zeta)}\in k(\zeta)\quad 
\end{gather*}
on $\bP^1(k)$, and the non-negative integer 
$\dep_{\hat{a}}\hat{\phi}:=[H_{\hat{\phi}}=0](\hat{a})\in\{0,1,\ldots,d\}$
is called the (algebraic) depth of $\hat{\phi}$ modulo $K^{\circ\circ}$ at each $\hat{a}\in\bP^1(k)$.
This coincides with the intrinsic depth $\dep_{\widevec{\xi_ga}}\tilde{\phi}_{\xi_g}$
(by the non-archimedean argument principle \eqref{eq:FKN} below together with
\cite[Lemma 3.17]{Faber13topologyI} and \cite[Proof of Corollary 2.11]{KN23}).

Applying the geometric invariant theory to the conjugation action of 
$\SL_2$ on $\bP^{2d+1}$,
let us denote by $(\bP^{2d+1})^{\operatorname{ss}}$ and
$(\bP^{2d+1})^{\operatorname{s}}$ respectively
the (GIT-)semistable and stable loci in $\bP^{2d+1}$
with respect to the $\SL_2$-linearized line bundle $\cO_{\bP^{2d+1}}(1)$.
By the Hilbert-Mumford numerical criterion,
we have $\hat{\phi}\in(\bP^{2d+1})^{\operatorname{ss}}(k)$
(resp.\ $\hat{\phi}\in(\bP^{2d+1})^{\operatorname{s}}(k)$)
if and only if
for every $\hat{a}\in\bP^1(k)$,
\begin{gather*}
\begin{cases}
 \dep_{\hat{a}}\hat{\phi}\le\dfrac{d+1}{2}  
 \,\bigl(\text{resp.}\,\le\dfrac{d}{2}\bigr)
\text{ and more strictly}\vspace*{5pt}\\
 \dep_{\hat{a}}\hat{\phi}<\dfrac{d}{2} 
 \,\bigl(\text{resp.}\,<\dfrac{d-1}{2}\bigr) 
 \text{ if also }\tilde{\phi}(\hat{a})=\hat{a}
\end{cases}
\end{gather*}
(see \cite{Silverman98} for more details, and also \cite[Theorem 3.5]{STW14} or \cite[Section 3]{DeMarcoQuad}).

Once the above background from GIT is at our disposal,
recalling the expression \eqref{eq:ordhyp} of $\ordRes_\phi$
in terms of $\hypRes_\phi$,
the coincidence of the intrinsic and algebraic depths, and the identification
of the left orbit space $\PGL(2,K^\circ)\backslash\PGL(2,K)$
with the totality $\sH^1_{\mathrm{II}}$ of type II points in $\sH^1$
by the bijection $[M]=\PGL(2,K^\circ)M\leftrightarrow M(\xi_g)$,
the latter half in Theorem \ref{th:slope} concludes
Theorem \ref{th:ss}.

\section{Proof of Theorem \ref{th:unicity} and its application to
a degeneration problem in complex dynamics}\label{sec:degenerating}

\begin{proof}[Proof of Theorem \ref{th:unicity}]
 Let $\phi\in K(z)$ be a rational function on $\bP^1$ of degree $d>1$,
 pick a point $\xi\in\sH^1$ 
such that $\phi^{-1}(\xi)\neq\{\xi\}$, and
 pick any weak limit 
\begin{gather*}
 \mu=\lim_{j\to\infty}
\frac{1}{d^{n_j}}
\sum_{\vec{v}\in T_\xi\sP^1}\bigl(\dep_{\vec{v}}(\widetilde{\phi^{n_j}})_\xi\bigr)\delta_{U(\vec{v})}
\end{gather*}
on $\sP^1/ S(\{\xi\})$
for some $(n_j)$ in $\bN$ tending to $\infty$ 
as $j\to\infty$,
which is a probability measure on $\sP^1/ S(\{\xi\})$ and is
written as the sum of at most countably many atoms.
Then by a diagonal argument, taking a subsequence of $(n_j)$ if necessary, 
Fatou's lemma yields
\begin{gather*}
\sum_{\vec{v}\in T_\xi\sP^1}
\biggl(\lim_{j\to\infty}
\frac{\dep_{\vec{v}}(\widetilde{\phi^{n_j}})_\xi}{d^{n_j}}\biggr)
\delta_{U(\vec{v})}
\le\mu
\end{gather*}
on $\sP^1/S(\{\xi\})$. 
On the other hand,
for every $\vec{v}\in T_\xi\sP^1(=(\sP^1/S(\{\xi\}))\setminus\{\{\xi\}\}$), 
using the inner regularity of the Radon measure $\nu_\phi$ on $\sP^1$ and
the equidistribution property \eqref{eq:equidist}, we have
\begin{multline*}
\bigl((\pi_\xi)_*\nu_\phi\bigr)(\{U(\vec{v})\})
=\nu_\phi(U(\vec{v}))
=\lim_{\epsilon\searrow 0}\bigl(\nu_\phi(U(\vec{v}))-\epsilon\bigr)\\
  \le\limsup_{j\to\infty}\frac{((\phi^{n_j})^*\delta_\xi)(U(\vec{v}))}{d^{n_j}}=\lim_{j\to\infty}\frac{\dep_{\vec{v}}(\widetilde{\phi^{n_j}})_\xi}{d^{n_j}}.
\end{multline*} 
Hence we have
$(\pi_\xi)_*\nu_\phi\le\mu$ and in turn $(\pi_\xi)_*\nu_\phi=\mu$
on $\sP^1/S(\{\xi\})$
also using $((\pi_\xi)_*\nu_\phi)(\{\{\xi\}\})(=\nu_\phi(\{\xi\}))=0$ 
under the assumption $\phi^{-1}(\xi)\neq\{\xi\}$. 
\end{proof}

We conclude this subsection with 
an application of Theorem \ref{th:unicity}
to a degeneration
problem mentioned in Subsection \ref{sec:deglim}
on a meromorphic family 
$f=(f_t)_{0<|t|\ll 1}\in(\cO(\bD)[t^{-1}])(z)\subset\bL(z)$ of complex rational functions $f_t$ 
on $\bP^1(\bC)$ 
of degree $d>1$ and in particular on
the family $(\mu_t)_{0<|t|\ll 1}$
of the (atomless) unique maximal entropy measures $\mu_t=\mu_{f_t}$ on $\bP^1(\bC)$
of the specializations $f_t$, where $K=\bL$. In this setting, $k_{\bL}\cong\bC$ so 
$T_{\xi_g}\sP^1_{\bL}\cong\bP^1(\bC)$
in a canonical way as mentioned in Section \ref{sec:intro}
so that the canonical equilibrium measure $\nu_f$ for $f\in\bL(z)$
on $\sP^1_{\bL}$ projects (or reduces) to the probability measure
$(\pi_{\xi_g})_*\nu_f$
on 
\begin{gather*}
 \sP^1_{\bL}/S(\{\xi\})
=\{\{\xi_g\}\}\sqcup T_{\xi_g}\sP^1_{\bL}
\cong\{\{\xi_g\}\}\sqcup\bP^1(\bC);
\end{gather*}
in particular, $(\pi_{\xi_g})_*\nu_f$ is
regarded as a probability measure on $\bP^1(\bC)$
written as
the sum of at most countably many atoms
if and only if $f^{-1}(\xi_g)\neq\{\xi_g\}$. 

\begin{fact}[Reduced pullbacks of measures]
For every $n\in\bN$, the operator $(f^n)_{\xi_g}^*$ is the reduced version at $\xi_g$ of the pullback operator $(f^n)^*$
 for measures on $\sP^1_{\bL}$
 mentioned in \S \ref{sec:intro}
 and is defined so that
 for any probability measure $\omega$ on $\sP^1_{\bL}/S(\{\xi_g,f^n(\xi_g)\})$,
$((f^n)_{\xi_g}^*\omega)/d^n$ is a probability
measure on $\sP^1_{\bL}/S(\{\xi_g\})$
and satisfies
 \begin{gather*}
 \int_{\sP^1_{\bL}/S(\{\xi_g\})}h\bigl((f^n)_{\xi_g}^*\omega\bigr)
 =\int_{\sP^1_{\bL}/S(\{\xi_g,f^n(\xi_g)\})}\bigl(((f^n)_{\xi_g})_*h\bigr)\omega 
 \end{gather*}
 for every test function $h$ on $\sP^1_{\bL}/S(\{{\xi_g}\})$, where for every $V\in S(\{\xi_g, f^n(\xi_g)\})$, we set
 \begin{gather*}
 \bigl(((f^n)_{\xi_g})_*h\bigr)(V):=\sum_{U\in S(\{\xi_g\})}
 \bigl((f^n)^*\delta_\xi\bigr)(U)\cdot h(U)
 \end{gather*}
fixing any $\xi\in V$, or equivalently that
\begin{multline}   
\label{eq:redpull}(f^n)_{\xi_g}^*\omega-\bigl(\deg_{\xi_g}(f^n)\bigr)\cdot\omega\bigl(\{\{f^n(\xi_g)\}\}\bigr)\delta_{\{\xi_g\}}
=\sum_{\vec{v}\in T_{\xi_g}\sP^1}\bigl(\dep_{\vec{v}}(\widetilde{f^n})_{\xi_g}\bigr)\delta_{U(\vec{v})}+\\ 
+\sum_{\vec{v}\in T_{\xi_g}\sP^1_{\bL}}
 \Biggl(\sum_{V\in S(\{\xi_g,f^n(\xi_g)\})}
 \Bigl(
 m_{\vec{v}}(f^n)\times\\
\times\begin{cases}
 1 & \text{if not only }V=U\bigl((f^n)_*\vec{v}\bigr)\text{ but, when }f^n(\xi_g)\neq\xi_g,\text{ also }
 (f^n)_*\vec{v}\neq\widevec{f^n(\xi_g)\xi_g};\\
 -1 & \text{if }f^n(\xi_g)\neq\xi_g,
	     (f^n)_*\vec{v}=\widevec{f^n(\xi_g)\xi_g},\text{ and }V\not\subset 
 U\bigl((f^n)_*\vec{v}\bigr)(=U\bigl(\widevec{f^n(\xi_g)\xi_g}\bigr));\\
 0 & \text{otherwise}
 \end{cases}\\
\Bigr)
 \omega(\{V\})\Biggr)\delta_{U(\vec{v})}\quad\text{on }\sP^1/S(\{\xi_g\}).
 \end{multline}
Here the well-definedness and the second 
(defining) equality \eqref{eq:redpull} 
of $(f^n)^*_{\xi_g}$ follow from
 the non-archimedean argument principle
 \begin{gather}
  \bigl(f^*\delta_{\xi'}\bigr)(U(\vec{u}))
 =s_{\vec{u}}(f)+
 \begin{cases}
  m_{\vec{u}}(f) & \text{if }U(f_*\vec{u})\ni\xi',\\
  0 & \text{otherwise} 
 \end{cases}\label{eq:FKN}
 \end{gather}
for every $\xi'\in\sP^1_{\bL}$ and
every direction 
$\vec{u}\in T_\xi\sP^1_{\bL}$
at every point $\xi\in\sH^1(\bL)$
in terms of the
directional local degree $m_{\vec{u}}(f)\in\{1,2,\ldots,\deg_\xi(f)\}$ and
the surplus local degree $s_{\vec{u}}(f)\in\{0,1,\ldots,\deg_\xi(f)\}$ of $f$; they satisfy the equalities
\begin{gather}
 \sum_{\vec{u}\in T_\xi\sP^1_{\bL}:\,f_*\vec{u}=\vec{w}}m_{\vec{u}}(f)=\deg_\xi(f)
\quad\text{for every }\vec{w}\in T_{f(\xi)}\sP^1_{\bL}\label{eq:directional}
\end{gather}
and
\begin{gather}
 \sum_{\vec{u}\in T_\xi\sP^1_{\bL}}s_{\vec{u}}(f)=d-\deg_\xi(f),\label{eq:surplus}
\end{gather} 
respectively.
\end{fact}

As an application of Theorem \ref{th:unicity},
we show the following.

\begin{mainth}[Improved degenerating limit theorem]\label{th:degenerating}
For every $f\in(\cO(\bD)[t^{-1}])(z)\subset\bL(z)$ of degree $d>1$ 
such that $f^{-1}(\xi_g)\neq\{\xi_g\}$,
letting $\mu_t=\mu_{f_t}$ be the unique maximal
entropy measure on $\bP^1(\bC)$ of $f_t$ for each $0<|t|\ll 1$, 
we have 
\begin{gather}
 \lim_{t\to 0}\mu_t
=\lim_{n\to\infty}\frac{1}{d^n}\sum_{\vec{v}\in T_{\xi_g}\sP^1_{\bL}}\bigl(\dep_{\vec{v}}(\widetilde{f^n})_{\xi_g}\bigr)\delta_{U(\vec{v})}
=(\pi_{\xi_g})_*\nu_f \label{eq:surpluslimit}
\end{gather}
weakly on $\bP^1(\bC)\cong T_{\xi_g}\sP^1_{\bL}$.
In particular, the weak limit
$\lim_{t\to 0}\mu_t$ is written as at most 
countably many atoms in $\bP^1(\bC)$.
\end{mainth}

Theorem \ref{th:degenerating} improves 
the degenerating limit assertion
$\lim_{t\to 0}\mu_t=(\pi_{\xi_g})_*\nu_f$ 
by DeMarco--Faber \cite[Theorem B]{DF14}\footnote{The proof in \cite{DF14} was based on a simplified statement \cite[Theorem 4.10]{DF14} (which is complemented in
\cite[Theorem A]{Okudegenerating}).} 
in that the first convergence assertion
in \eqref{eq:surpluslimit} is a new input
(which has been seen
in a purely complex analytic and
more general setting \cite[Propositions 4.3 and 4.6]{KN24}). Based on Theorem \ref{th:unicity}
(i.e., the second convergence assertion in \eqref{eq:surpluslimit}),
we give a much simpler and more natural proof 
of $\lim_{t\to 0}\mu_t=(\pi_{\xi_g})_*\nu_f$
than the previous one;
we dispense with an {\em a priori}
verification of purely atomicness of any weak limit point of $(\mu_t)_{0<|t|\ll 1}$ on $\bP^1(\bC)$,
and indeed obtain
the purely atomicness of the weak limit $\lim_{t\to 0}\mu_t$ as a consequence of \eqref{eq:surpluslimit}. As a byproduct of the proof,
a new equidistribution assertion is also 
obtained (see Proposition \ref{th:equidistr} 
below).

Recall that, when $K=\bL$, 
$E(f)$ consists of at most two points in $\bP^1(\bL)$, and replacing $f$ with $f^2$ if necessary,
we assume that $f^{-1}(a)=\{a\}$ 
for every $a\in E(f)$
without loss of generality. Moreover, we also have
$E(f)\subset \bP^1(\cO(\bD)[t^{-1}])$
replacing $t$ with some power
of it (i.e., by a base change of $\bL$)
and scaling $t$ complex affinely 
around the origin $t=0$ if necessary.

\begin{notation}
For any probability measure $m$ on $\bP^1(\bC)$,
let us denote the purely atomic part of $m$ by
\begin{gather*}
 m^*:=\sum_{w\in\bP^1(\bC)}m(\{w\})\delta_w.
\end{gather*}
\end{notation}

\begin{proof}[Proof of the first
convergence in \eqref{eq:surpluslimit}]
 Under the assumption that $f^{-1}(\xi_g)\neq\{\xi_g\}$,
 pick any weak limit point 
 \begin{gather*}
 \mu=\lim_{j\to\infty}\mu_{t_j}\quad\text{on }\bP^1(\bC) 
 \end{gather*}
 of $(\mu_t)_{0<|t|\ll 1}$ for some sequence $(t_j)$ in $\bD\setminus\{0\}$ tending to $0$ as $j\to\infty$.
For a notational simplicity, from now on, we write
\begin{align*}
 S_g&=S(\{\xi_g\})\quad\text{and}\\
S_n&=S(\{\xi_g,f^n(\xi_g)\})\quad\text{for every }n\in\bN. 
\end{align*}

If $f^n(\xi_g)=\xi_g$ for some $n\in\bN$, then 
we replace $f$ with $f^n$ without loss of 
generality. Then from the above
weak limit
$\mu$,
which satisfies a degenerated balanced property
(\cite[Theorem 2.4]{DF14}), 
a probability measure 
$\omega_\mu$ on 
$\bP^1/S_g$ is constructed so that for every
$\vec{v}\in T_{\xi_g}\sP^1_{\bL}$,
$\omega_\mu(U(\vec{v}))=\mu(\{\hat{a}\})$
writing $\vec{v}=\widevec{\xi_ga}$ by some 
$a\in\bP^1(\bL)$ (so $\hat{a}\in\bP^1(\bC)\cong T_{\xi_g}\sP^1_{\bL}$)
and that for every $n\in\bN$,
\begin{gather}
\mu^*\equiv
\frac{(f^n)_{\xi_g}^*\omega_\mu-(\deg_{\xi_g}(f^n))\cdot
 \omega_\mu\bigl(\{\{\xi_g\}\}\bigr)\delta_{\{\{\xi_g\}\}}}{d^n}\label{eq:redbalfixed}
\end{gather}
on $\bP^1(\bC)\cong(\sP^1_{\bL}/S_g)\setminus\{\{\xi_g\}\}$. Hence 
we are done by Theorem \ref{th:unicity}
since for {\em any} probability
measure $\omega$ on $\sP^1_{\bL}/S_n$, 
by \eqref{eq:redpull},
\begin{gather*}
 (f^n)_{\xi_g}^*\omega-\bigl(\deg_{\xi_g}(f^n)\bigr)\cdot
 \omega\bigl(\{\{f^n(\xi_g)\}\}\bigr)\delta_{\{\xi_g\}}
\ge \sum_{\vec{v}\in T_{\xi_g}\sP^1}\bigl(\dep_{\vec{v}}(\widetilde{f^n})_{\xi_g}\bigr)\delta_{U(\vec{v})}
\end{gather*}
on $\sP^1_{\bL}/S(\{\xi_g\})$.

From now on, suppose that $f^n(\xi_g)\neq\xi_g$
for any $n\in\bN$.
Then for every $n\in\bN$, there is $A_n\in\PGL(2,\cO(\bD)[t^{-1}])\subset\PGL(2,\bL)$,
which is unique up to a postcomposition by
$\PGL(2,\cO(\bD))$
(see \cite[Proof of Lemma 2.1]{DF14}),
such that
\begin{gather*}
 A_n^{-1}(\xi_g)=f^n(\xi_g). 
\end{gather*}
By a diagonal argument, taking a subsequence of $(t_j)$ if necessary, 
for every $n\in\bN$, there is also the weak limit
\begin{gather*}
  \mu^{(n)}:=\lim_{j\to\infty}(A_n)_*\mu_{t_j}
\quad\text{on }\bP^1(\bC).
\end{gather*}
For every $n\in\bN$,
from the pair $(\mu,\mu^{(n)})$,
which satisfies a degenerated balanced property
(\cite[Theorem 2.4]{DF14}), 
of those two probability measures on $\bP^1(\bC)$,
a probability measure $\omega_\mu^{(n)}$
on $\sP^1_{\bL}/S_n$ 
is constructed 
so that for every $\vec{v}\in(T_{\xi_g}\sP^1_{\bL})\setminus\{\widevec{\xi_gf^n(\xi_g)}\}$,
$\omega_\mu^{(n)}(U(\vec{v}))=\mu(\{\hat{a}\})$ 
writing $\vec{v}=\widevec{\xi_ga}$ by some 
$a\in\bP^1(\bL)$ (so $\hat{a}\in\bP^1(\bC)\cong T_{\xi_g}\sP^1_{\bL}$), 
that
for every $\vec{w}\in(T_{f^n(\xi_g)}\sP^1_{\bL})\setminus\{\widevec{f^n(\xi_g)\xi_g}\}$,
$\omega_\mu^{(n)}(U(\vec{w}))=\mu^{(n)}(\{\hat{b}\})$
writing $(A_n)_*\vec{w}=\widevec{\xi_gb}$
by some 
$b\in\bP^1(\bL)$ (so $\hat{b}\in\bP^1(\bC)\cong T_{\xi_g}\sP^1_{\bL}$), 
that 
\begin{align*}
 \omega_\mu^{(n)}\bigl(\{\{\xi_g\}\}\bigr)&=1-\mu^*(\bP^1(\bC))\quad\text{and}\\ 
\omega_\mu^{(n)}\bigl(\{\{f^n(\xi_g)\}\}\bigr)&=1-(\mu^{(n)})^*(\bP^1(\bC)),
\end{align*}
and that
\begin{gather}
\mu^*\equiv
\frac{(f^n)_{\xi_g}^*\omega_\mu^{(n)}-(\deg_{\xi_g}(f^n))\cdot
 \omega_\mu^{(n)}\bigl(\{\{f^n(\xi_g)\}\}\bigr)\delta_{\{\{\xi_g\}\}}}{d^n}\label{eq:reducedbalanced}
\end{gather}
on $\bP^1(\bC)\cong(\sP^1_{\bL}/S_g)\setminus\{\{\xi_g\}\}$
(for the constructions
of $\omega_\mu$ from $\mu$
in the previous case and
of $\omega_\mu^{(n)}$ from $(\mu,\mu^{(n)})$ here
respectively
and their reduced balanced properties
\eqref{eq:redbalfixed} and \eqref{eq:reducedbalanced}, see \cite[\S4]{Okudegenerating}, which is more direct than the {\em transfer principle} in \cite[Proposition 5.1(1)]{DF14}).

In the case of $\deg_{\xi_g}(f^n)=o(d^n)$
as $n\to\infty$, we are done 
by Theorem \ref{th:unicity} 
since for {\em any} probability
measures $\omega_n$ on $\sP^1_{\bL}/S_n$, 
$n\gg 1$,
by \eqref{eq:redpull} and \eqref{eq:directional},
\begin{multline*}
 (f^n)_{\xi_g}^*\omega_n-\bigl(\deg_{\xi_g}(f^n)\bigr)\cdot
 \omega_n\bigl(\{\{f^n(\xi_g)\}\}\bigr)\delta_{\{\xi_g\}}\\
\ge \sum_{\vec{v}\in T_{\xi_g}\sP^1}\bigl(\dep_{\vec{v}}(\widetilde{f^n})_{\xi_g}\bigr)\delta_{U(\vec{v})}
-O(\deg_{\xi_g}(f^n))\quad\text{as }n\to\infty
\end{multline*}
on $\sP^1_{\bL}/S(\{\xi_g\})$.

Finally, suppose that
$\deg_{\xi_g}(f^n)\neq o(d^n)$ as $n\to\infty$.
Under this assumption,
the sequence $(f^n(\xi_g))_{n\gg 1}$ 
is in the Berkovich maximal ramification locus 
\begin{gather*}
 \mathsf{R}_{\max}(f):=\{\xi\in\sP^1_{\bL}:\deg_\xi(f)=d\} 
\end{gather*}
of $f$.
Then there is $a_0\in E(f)(\subset\mathsf{R}_{\max}(f))$ such that 
$A_n^{-1}(\xi_g)=f^n(\xi_g)\to a_0$ 
as $n\to\infty$ 
and, moreover, that for $n\gg 1$, 
$f^{n+1}(\xi_g)\in[f^n(\xi_g),a_0]\setminus\{f^n(\xi_g),a_0\}\subset[\xi_g,a_0]$ (for more details, see \cite[P.\ 25]{DF14}); then the $A_n$ can be chosen so that $A_n(a_0)\equiv a_0$ and that
($A_n(\xi_g)\neq\xi_g$ and) we write
\begin{gather*}
 \widevec{\xi_gA_n(\xi_g)}=\widevec{\xi_gb_n}\neq\widevec{\xi_ga_0}
\end{gather*}
by some $b_n\in\bP^1(\bL)$
(indeed, in some projective coordinate,
$a_0=\infty$, $b_n\equiv 0$, and $A_n$ is a scaling around $0\in K$ \cite[\S 6(a)]{Okudegenerating}).

For $n\gg 1$,
by \eqref{eq:reducedbalanced},
\eqref{eq:redpull}, and \eqref{eq:directional},
we have
\begin{gather*}
\mu^*
\ge\frac{1}{d^n}\sum_{\vec{v}\in T_{\xi_g}\sP^1}\bigl(\dep_{\vec{v}}(\widetilde{f^n})_{\xi_g}\bigr)\delta_{U(\vec{v})}
-\omega_\mu^{(n)}\bigl(\{V\in S_n:V\not\subset U(\widevec{f^n(\xi_g)\xi_g})\}\bigr)
\times O(1)
\quad\text{as }n\to\infty
\end{gather*}
on $\sP^1_{\bL}/S(\{\xi_g\})$.
We claim that
\begin{gather}
\lim_{n\to\infty}\omega_\mu^{(n)}\bigl(\{V\in S_n:V\not\subset U(\widevec{f^n(\xi_g)\xi_g})\}\bigr)=0,
\end{gather}
which will conclude the proof 
by Theorem \ref{th:unicity} in this final case;
for, on one hand, recall that
$\supp(\mu^{(n)})\subset\bP^1(\bC)\setminus\{\widehat{a_0}\}$
(see \cite[\S 6, Claim 2]{Okudegenerating}),
which yields
\begin{gather*}
 \omega_\mu^{(n)}\bigl(\{U(\widevec{f^n(\xi_g)a_0})\}\bigr)=\mu^{(n)}(\{\widehat{a_0}\})\equiv 0
\end{gather*}
for $n\gg 1$.\footnote{The original argument \cite[Corollary 5.3 and Proof of Theorem B]{DF14} to see $\mu^{(n)}(\{\widehat{a_0}\})\equiv 0$ seems difficult to read. This part is complemented in \cite[Claim 2 in \S6]{Okudegenerating} by a careful comparison of the behaviour of $f$ near the exceptional point $a_0(=\infty$ in some projective coordinate) in the $t$-adic norm $|\cdot|_r$ on $\bL$ with that of each specialization $f_t$ near the specialization $(a_0)_t\equiv\infty\in\bP^1(\bC)$ in the Euclidean norm $|\cdot|$ on $\bC$ for $0<|t|\ll 1$. This argument is conceptually similar to hybrid analysis (see e.g.\ \cite{Favre16}). In \cite[\S 3, using Proposition 3.5]{KN24}, a geometric and purely complex analytic argument is also developed even for degenerating sequences of complex rational functions.}
On the other hand, for $n\gg 1$, we note that
\begin{gather*}
 \omega_\mu^{(n)}\bigl(\{V\in S_n:V\not\subset U(\widevec{f^n(\xi_g)a_0})\sqcup U(\widevec{f^n(\xi_g)\xi_g})\}\bigr)
=\mu^{(n)}\bigl(\bP^1(\bC)\setminus\{\widehat{a_0},\widehat{b_n}\}\bigr)(\ge 0),
\end{gather*}
that $\lim_{j\to\infty}(A_n^{-1})_{t_j}=\widehat{a_0}$ locally uniformly 
on $\bP^1(\bC)\setminus\{\widehat{b_n}\}$,
and moreover that
\begin{gather*}
 \lim_{j\to\infty}(A_{n+1}\circ A_n^{-1})_{t_j}=\widehat{b_{n+1}}
\end{gather*}
locally uniformly on 
$\bP^1(\bC)\setminus\{\widehat{a_0}\}$.
Fix any $m>N\gg 1$,
and fix any family $(K_n)_{N\le n\le m}$
where  for each $n$,
$K_n$ is a compact subset
in $\bP^1(\bC)\setminus\{\widehat{a_0},\widehat{b_n}\}$. Then for every $n\in\{N,N+1,\ldots,m-1\}$, 
we inductively have 
\begin{gather*}
 \biggl((A_{n+1})_{t_j}\Bigl(\bigcup_{\ell=N}^n(A_\ell^{-1})_{t_j}(K_\ell)\Bigr)\biggr)
\cap K_{n+1}=\emptyset 
\end{gather*}
for $j\gg 1$, 
so that for $j\gg 1$,
the compact subsets $(A_n^{-1})_{t_j}(K_n)$
in $\bP^1(\bC)$, $N\le n\le m$, are 
mutually disjoint 
(cf.\ the final argument in \cite[\S2.5, Proof of Theorem A]{DF14}). Then
using the outer regularity of $\mu^{(n)}$, 
we compute
\begin{align*}
 1\ge\liminf_{j\to\infty}\sum_{N\le n\le m}
\mu_{t_j}((A_n^{-1})_{t_j}(K_n))
=\sum_{N\le n\le m}\liminf_{j\to\infty}
\bigl(((A_n)_{t_j})_*\mu_{t_j}\bigr)(K_n)
\ge\sum_{N\le n\le m}\mu^{(n)}(K_n),
\end{align*}
which in turn, also by the inner regularity of $\mu^{(n)}$, yields
\begin{gather*}
\sum_{n\ge N}\mu^{(n)}\bigl(\bP^1(\bC)\setminus\{\widehat{a_0},\widehat{b_n}\}\bigr)
=\limsup_{m\to\infty}\sum_{N\le n\le m}\mu^{(n)}\bigl(\bP^1(\bC)\setminus\{\widehat{a_0},\widehat{b_n}\}\bigr)\le 1<\infty.
\end{gather*}
Hence the claim holds.
\end{proof}

For a future reference, we conclude with
the following proposition, which is a byproduct of the arguments in the first two cases in the above proof.

\begin{proposition}\label{th:equidistr}
For every rational function 
$\phi\in K(z)$ of degree $d>1$,
if $\deg_{\xi_g}(\phi^n)=o(d^n)$ as $n\to\infty$, then for any  probability
measures $\omega_n$ on $\sP^1/S(\{\xi_g,\phi^n(\xi_g)\})$, $n\gg 1$,
we have the weak convergence
\begin{gather*}
\lim_{n\to\infty}\frac{(\phi^n)^*_{\xi_g}\omega_n}{d^n}
=(\pi_{\xi_g})_*\nu_\phi\quad\text{on }\sP^1/S(\{\xi_g\}).
\end{gather*}
\end{proposition}

\begin{acknowledgement}
 This research was partially supported by JSPS Grant-in-Aid for Scientific Research (C), 23K03129.
\end{acknowledgement}

\end{document}